\documentclass[11pt,a4paper]{article}

\usepackage{epsf,epsfig,amsfonts,amsgen,amsmath,amstext,amsbsy,amsopn,amsthm
}
\usepackage{amsmath,times,mathptmx}
\usepackage{enumitem}
\usepackage{amsfonts,amsthm,amssymb}
\usepackage{amsfonts}
\usepackage{graphics}
\usepackage{latexsym,bm}
\usepackage{amsfonts,amsthm,amssymb,bbding}
\usepackage{indentfirst}
\usepackage{graphicx}
\usepackage{color}
\usepackage[colorlinks=true,anchorcolor=blue,filecolor=blue,linkcolor=blue,urlcolor=blue,citecolor=blue]{hyperref}
\usepackage{float}
\usepackage{tikz}

\DeclareMathAlphabet{\mathcal}{OMS}{cmsy}{m}{n}
\DeclareSymbolFont{largesymbols}{OMX}{cmex}{m}{n}

\newtheorem{thm}{Theorem}

\newtheorem{lemma}{Lemma}
\newtheorem{false statement}{False statement}

\theoremstyle{definition}
\newtheorem{definition}{Definition}
\newtheorem{claim}{Claim}
\newtheorem{claimm}{Claim}
\newtheorem{claimmm}{Claim}
\newtheorem{claimmmm}{Claim}

\newtheorem{remark}
{Remark}

\newtheorem{case}{Case}
\newtheorem{casee}{Case}

\newtheorem{problem}{Problem}

\begin{document}

\title{Spectral extremal problems on outerplanar and planar graphs
\footnote{Supported by NSFC (Nos. 12361071 and 11901498).}}
\author{{Xilong Yin, Dan Li}\thanks{Corresponding author. E-mail: ldxjedu@163.com.}\\
{\footnotesize College of Mathematics and System Science, Xinjiang University, Urumqi 830046, China}}
\date{}

\maketitle {\flushleft\large\bf Abstract:}
Let $\emph{spex}_{\mathcal{OP}}(n,F)$ and $\emph{spex}_{\mathcal{P}}(n,F)$ be the maximum spectral radius over all $n$-vertex $F$-free outerplanar graphs and planar graphs, respectively. 
Define $tC_l$ as $t$ vertex-disjoint $l$-cycles, $B_{tl}$ as the graph obtained by sharing a common vertex among $t$ edge-disjoint $l$-cycles 
and $(t+1)K_{2}$ as the disjoint union of $t+1$ copies of $K_2$.
In the 1990s, Cvetković and Rowlinson conjectured $K_1 \vee P_{n-1}$ maximizes spectral radius in outerplanar graphs on $n$ vertices, while Boots and Royle (independently, Cao and Vince) conjectured $K_2 \vee P_{n-2} $ does so in planar graphs.
Tait and Tobin [J. Combin. Theory Ser. B, 2017] determined the fundamental structure as the key to confirming these two conjectures for sufficiently large $n.$
Recently, Fang et al. [J. Graph Theory, 2024] characterized the extremal graph with $\emph{spex}_{\mathcal{P}}(n,tC_l)$ in planar graphs by using this key.
In this paper, we first focus on outerplanar graphs and adopt a similar approach to describe the key structure of the connected extremal graph with $\emph{spex}_{\mathcal{OP}}(n,F)$, where
$F$ is contained in $K_1 \vee P_{n-1}$ but not in $K_{1} \vee ((t-1)K_2\cup(n-2t+1)K_1)$.
Based on this structure, we determine $\emph{spex}_{\mathcal{OP}}(n,B_{tl})$ and $\emph{spex}_{\mathcal{OP}}(n,(t+1)K_{2})$ along with their unique extremal graphs for all $t\geq1$, $l\geq3$ and large $n$.
Moreover, we further extend the results to planar graphs, characterizing the unique extremal graph with $\emph{spex}_{\mathcal{P}}(n,B_{tl})$ for all $t\geq3$, $l\geq3$ and large $n$.

\vspace{0.1cm}
\begin{flushleft}
\end{flushleft}

\section{Introduction}
For a graph $G$, let $A(G)$ be the adjacency matrix of the graph $G$, and let $\rho(G)$ be its spectral radius(i.e., the maximum modulus of eigenvalues of $A(G)$).
Given a graph family $\mathcal{F}$, a graph is said to be $\mathcal{F}$-free if it does not contain any $F\in\mathcal{F}$ as a subgraph.
When $\mathcal{F}=\{F\}$, we write $F$-free instead of $\mathcal{F}$-free. Let $\emph{spex}(n,F)$ be the maximum spectral radius over all $F$-free graphs of order $n$, and SPEX$(n,F)$ be the family of $F$-free graphs of order $n$ with spectral radius equal to $\emph{spex}(n,F).$

The \emph{Brualdi-Solheid problem} \cite{R.A. Brualdi}, a classic problem in spectral graph theory, seeks the graph with the maximum spectral radius within a given class. 
In 2010, Nikiforov \cite{V. Nikiforov-4} proposed a variation of Brualdi-Soheid problem which asks: given a graph $F$, what is the maximum spectral radius of an $F$-free graph of order $n$?
In the past decades, many researchers have studied this variation for forbidding various subgraphs, such as complete graphs $K_r$ \cite{B. Bollobás,H. Wilf}, independent edges $M_{k}$ \cite{L.H. Feng}, paths $P_{k}$ \cite{V. Nikiforov-4}, friendship graphs $F_{k}$ \cite{V. Nikiforov-3}, cycles \cite{V. Nikiforov-1,V. Nikiforov-2,M.Q. Zhai-1,M.Q. Zhai-2}, wheels $W_{k}$ \cite{Y.H. Zhao,S. Cioab˘a-2},
stars $\bigcup _{i= 1}^{k}S_{a_{i}}$ \cite{M.Z. Chen-1} and linear forests $\bigcup _{i= 1}^{k}P_{a_{i}}$ \cite{M.Z. Chen-2}. For more about this topic, we refer readers to three surveys \cite{M.Z. Chen-3,Y.T. Li,V. Nikiforov-3}.
Cioabă, Desai and Tait \cite{S. Cioab˘a-1} confırmed a famous conjecture of Nikiforov \cite{V. Nikiforov-4} about charactering the extremal graphs in SPEX$(n,C_{2k})$ for any $k\geq3$ and sufficiently large $n$.
Very recently, Fang, Lin, Shu and Zhang \cite{L.F. Fang-1} characterized the extremal graphs with $\emph{spex}(n,F)$ for various specific trees $F$.

Research on spectral radius in planar graphs is productive, originating with a question posed by Schwenk and Wilson \cite{A. J. Schwenk}: What about the eigenvalues of a planar graph?
This led to the extension of the Brualdi-Solheid problem to planar graphs.
Given two graphs $G_1$ and $G_2$,
$G_1\vee G_2$ denotes their \emph{join} which is the graph obtained by adding all possible edges between $G_1$ and $G_2.$
Boots and Royle \cite{B. N. Boots}, and independently, Cao and Vince \cite{D. Cao}, conjectured that $K_2 \vee P_{n-2}$ attains the maximum spectral radius among all planar graphs on $n\geq9$ vertices. 
In 2017, Tait and Tobin \cite{M. Tait} confirmed that the conjecture is true for sufficiently large $n$. Up to now, for small $n$, this conjecture is still open.
The study of the Brualdi-Solheid problems for planar graphs with specific forbidden subgraphs is also becoming a fruitful topic in recent times.
In 2022, Zhai and Liu \cite{M.Q. Zhai-2} determined the extremal graphs in SPEX$_\mathcal{P}(n,\mathcal{F})$ when $\mathcal{F}$ is the family of $k$ edge-disjoint cycles.
In 2023, Fang, Lin and Shi \cite{L.F. Fang-2} characterized the extremal graphs in SPEX$_\mathcal{P}(n,tC_\ell)$ and SPEX$_\mathcal{P}(n,tC)$, where $tC_\ell$ is $t$ vertex-disjoint $\ell$-cycles, and $tC$ is the family of $t$ vertex-disjoint cycles without length restriction.
Very recently, for sufficiently large $n$, Wang, Huang and Lin \cite{X.L. Wang} determined the extremal graphs in $\operatorname{SPEX}_{\mathcal{P}}(n,W_{k}),\operatorname{SPEX}_{\mathcal{P}}(n,F_{k})$ and SPEX$_\mathcal{P}(n,(k+1)K_{2})$, where $W_{k},F_{k}$ and $(k+1)K_{2}$ are the wheel graph of order $k$, the friendship graph of order $2k+1$ and the disjoint union of $k+1$ copies of $K_2$, respectively.
Similarly, Zhang and Wang \cite{H.R. Zhang} characterized the unique extremal planar graph with the maximum spectral radius among $C_{l,l}$-free planar graphs and $\theta_k$-free planar graphs on $n$ vertices,
where $C_{l,l}$ be a graph obtained from $2C_l$ such that the two cycles share a common vertex and $\theta_k$ are the Theta graph of order $k \geq 4.$

In fact, even earlier than the Boots–Royle–Cao–Vince conjecture, Cvetković and Rowlinson \cite{D. Cvetković} had already considered the Brualdi-Solheid problem in outerplanar graphs and proposed a conjecture on the maximum spectral radius of outerplanar graphs in 1990.
In 2017, Tait and Tobin \cite{M. Tait} confirmed that the conjecture is true for sufficiently large $n$.
Recently, this conjecture has been completely confirmed by Lin and Ning \cite{H.Q. Lin}.
Inspired by the aforementioned results on planar graphs, we naturally focus on the Brualdi-Solheid problems for outerplanar graphs with specific forbidden subgraphs.
These results are almost always obtained by determining the fundamental structure of the extremal graph with $\emph{spex}_{\mathcal{P}}(n,F)$ for certain specific graphs $F$.
With regard to the outerplanar spectral extremal graph, we raise the following analogous question:

\begin{problem}\label{p3}
What is the fundamental structure of outerplanar spectral extremal graph that do not contain a subgraph $F$ with certain characteristics?
\end{problem}

In this paper, we first answer the question.
Our result provide a structural theorem for connected extremal graphs with $\emph{spex}_\mathcal{OP}(n,F)$
under the condition that $F$ is an outerplanar graph contained in $K_1 \vee P_{n-1}$ but not in $K_{1} \vee ((t-1)K_2\cup(n-2t+1)K_1)$ and $n$ is sufficient large relative to the order of $F$,
which extends a result of Fang, Lin and Shi \cite{L.F. Fang-2} to outerplanar graphs.
As usual, we denote by $V(G)$ the vertex set and $E(G)$ the edge set. For $S \subseteq V(G)$, $G[S]$ denotes the subgraph of $G$ induced by $S$.
For any vertex $v \in V(G)$, $N_G(v)$ denotes the set of neighbors of $v$ in $G$.

\begin{thm}\label{thm1}
 Let $F$ be an outerplanar graph contained in $K_{1} \vee P_{n-1}$ but not contained in $K_{1} \vee ((t-1)K_2\cup(n-2t+1)K_1)$,
where $1 \leq t \leq \frac{n-1}{2} $ and $n\geq max\{1.27\times 10^7,[(64+|V(F)|)^\frac{1}{2}+8]^2\}$. Suppose that $G$ is a connected extremal graph with $\textit{spex}_{\mathcal{OP}}(n,F)$ and $X$ 
is the positive eigenvector of $\rho(G)$ with $\emph{max}_{v\in V(G)}x_v = 1 $.
Then there exists a vertice $u \in V(G)$
   such that $N_G(u)=V(G)\setminus\{u\}$ and $x_{u}=1$, in especially,

 \begin{description}
   \item[(i)]  $G$ contains a copy of $K_{1,n-1}.$
   \item[(ii)] The subgraph $G[N_G(u)]$ is a disjoint union of some paths.
 \end{description}
\end{thm}
\vspace*{2mm}
Recall that $B_{tl}$ be a graph obtained by sharing a common vertex among $t$ edge-disjoint $l$-cycles and $(t+1)K_{2}$ be a disjoint union of $t+1$ copies of $K_2$.
As applications of Theorem \ref{lm1}, we consider the Brualdi-Solheid problem for $B_{tl}$-free and $(t+1)K_{2}$-free outerplanar graphs, respectively.

\begin{problem}\label{p1}
Which outerplanar graphs have the $\emph{spex}_{\mathcal{OP}}(n,B_{tl})$ and $\emph{spex}_{\mathcal{OP}}(n,(t+1)K_{2})$?
\end{problem}
If $B_{tl}$ is not contained in $K_1\vee P_{n-1}$,
then $K_1 \vee P_{n-1}$ is $B_{tl}$-free. Note that $K_1 \vee P_{n-1}$ is the spectral extremal graph over all outerplanar graph for sufficiently large $n$.
This implies that $K_1 \vee P_{n-1}$ is the extremal graph to $spex_{\mathcal{OP}}(n,B_{tl})$ for $B_{tl}$ is not contained in $K_1 \vee P_{n-1}$ and sufficiently large $n.$ 

\begin{definition}
For three positive integers $n,n_1,n_2$ with $n\geq n_1>n_2 \geq 1$. Define $H(n_1, n_2)$ as follows:
\[  
H_{\mathcal{OP}}(n_1, n_2) = \begin{cases}   
P_{n_1} \cup \frac{n-1-n_1}{n_2} P_{n_2}, & \text{if } n_2|(n-1-n_1); \\  
P_{n_1} \cup \left\lfloor \frac{n-1-n_1}{n_2} \right\rfloor P_{n_2} \cup P_{n-1-n_1-\left\lfloor \frac{n-1-n_1}{n_2} \right\rfloor n_2}, & \text{otherwise.}  
\end{cases}  
\]
\end{definition}

In this paper, we give the answers to the Problem \ref{p1}. 
The main outcomes of our investigation are as follows:

\begin{thm}\label{thm2}
For integers $l\geqslant3$ , $t\geqslant1$ and $n\geq max\{1.27\times 10^7,6.5025\times2^{(t-1)(l-1)+l}\}$, the graph $B$ is the extremal graph to $\textit{spex}_{\mathcal{OP}}(n,B_{tl})$,
where 
 \begin{description}
   \item[(i)] $B\cong K_1\vee H_{\mathcal{OP}}(l-2,l-2)$ is the extremal graph to $\textit{spex}_{\mathcal{OP}}(n,B_{tl})$ for $t=1$ and $l\geq 3$.
   \item[(ii)] $B\cong K_1\vee H_{\mathcal{OP}}(tl-t-1,l-2)$ is the extremal graph to $\textit{spex}_{\mathcal{OP}}(n,B_{tl})$ for $t\geqslant2$ and $l\geq 3$.
 \end{description}
\end{thm}

\begin{thm}\label{thm3}
For integers $t\geqslant1$ and $n\geqslant N+\frac{3}{2}+3\sqrt{N-\frac{7}{4}}$ where $N=max\{1.27\times 10^7,6.5025\times2^{2t+1}\}$, the graph $M_t$ is the extremal graph to $\textit{spex}_{\mathcal{OP}}(n,(t+1)K_{2})$, where
 \begin{description}
   \item[(i)] $M_t \cong K_{1,n-1}$ is the extremal graph to $\textit{spex}_{\mathcal{OP}}(n,(t+1)K_{2})$ for $t=1$;
   \item[(ii)] $M_t \cong K_1\vee H_{\mathcal{OP}}(2t-1,1)$ is the extremal graph to $\textit{spex}_{\mathcal{OP}}(n,(t+1)K_{2})$ for $t\geqslant2$.
 \end{description}
\end{thm}

\vspace*{2mm}
Surprisingly, inspired by the above results, we consider whether we can extend the Brualdi-Solheid problem for $B_{tl}$-free graphs on the outerplanar graphs back to on the planar graphs.
\begin{problem}\label{p2}
Which planar graphs have the $\emph{spex}_{\mathcal{P}}(n,B_{tl})$?
\end{problem}

\begin{definition}
For three positive integers $n,n_1,n_2$ with $n\geq n_1>n_2 \geq 1$. Define $H(n_1, n_2)$ as follows:
\[  
H_{\mathcal{P}}(n_1, n_2) = \begin{cases}   
P_{n_1} \cup \frac{n-2-n_1}{n_2} P_{n_2}, & \text{if } n_2|(n-2-n_1); \\  
P_{n_1} \cup \left\lfloor \frac{n-2-n_1}{n_2} \right\rfloor P_{n_2} \cup P_{n-2-n_1-\left\lfloor \frac{n-2-n_1}{n_2} \right\rfloor n_2}, & \text{otherwise.}  
\end{cases}  
\]
\end{definition}

\begin{remark}\label{rk111}
There are some known results for $\textit{spex}_\mathcal{P}(n,B_{tl})$ and related extremal graphs.
\begin{enumerate}[label=(\Roman*)]  
    \item If $t = 1$, then $B_{tl} \cong C_{l}$. 
     \begin{enumerate}[label=(\alph*)]
         \item If $l = 3$, then we state that SPEX$_{\mathcal{P}}(n, B_{tl})$ = SPEX$_{\mathcal{P}}(n, C_{3}) = K_{2, n-2}$. Clearly, we deduce that $G \cong K_{2, n-2}$ since $G$ is $C_3$-free; otherwise, adding any edge would form a $C_3$, leading to a contradiction.
         \item If $l = 4$, then we assert that SPEX$_{\mathcal{P}}(n, B_{tl})$ = SPEX$_{\mathcal{P}}(n, C_{4}) = J_n$, where $J_n$ is derived from $K_1 \vee (n-1)K_1$ by embedding a maximum matching within its independent set. Nikiforov \cite{V. Nikiforov-2} and Zhai and Wang \cite{M.Q. Zhai-1} independently proved that $J_n$ is the extremal graph for $\textit{spex}(n, C_4)$ for both odd and even $n$. Noting that $J_n$ is planar, thus implying $G \cong J_n$.
         \item For $l \geq 5$, Fang, Lin, and Shi \cite{L.F. Fang-2} established that SPEX$_{\mathcal{P}}(n, B_{tl})$ = SPEX$_{\mathcal{P}}(n, C_{l}) = K_{2} \vee H_{\mathcal{P}}\left(\lfloor\frac{l-3}{2}\rfloor, \lceil\frac{l-3}{2}\rceil\right)$.
\end{enumerate} 
    \item If $t = 2$ and $l \geq 3$, then $B_{tl} \cong C_{l,l}$. Zhang and Wang \cite{H.R. Zhang} showed that SPEX$_{\mathcal{P}}(n, C_{l,l}) = K_2 \vee H_{\mathcal{P}}(l-2, l-2)$.  
    \item If $t \geq 3$ and $l = 3$, then $B_{tl} \cong F_{t}$. Wang, Huang and Lin \cite{X.L. Wang} determined that SPEX$_{\mathcal{P}}(n, F_{t}) = K_2 \vee H_{\mathcal{P}}(2t-3, 1)$.
\end{enumerate}
\end{remark}

\begin{thm}\label{thm4}
For integers $t\geqslant3$, $l\geqslant3$, $n\geqslant max\{2.67\times9^{17},10.2\times2^{(t-1)(l-1)+l-3}+2,
\frac{625}{32} \lfloor\frac{l-3}{2}\rfloor^2 + 2\}$,
then \emph{SPEX}$_{\mathcal{P}}(n,B_{tl}) = K_2\vee H_{\mathcal{P}}(tl-t-l,l-2)$.
\end{thm}

\vspace*{2mm}
An outline of the paper is as follows.
In section \ref{sec2}, we prove the Theorem \ref{thm1}, which will be frequently used in the following.
In Section \ref{sc3}, we give the proofs of Theorems \ref{thm2} and \ref{thm3} respectively.
In Section \ref{sc4}, we further extend the results to planar graph, determining the extremal graph with $\textit{spex}_{\mathcal{P}}(n,B_{tl})$ by proofing Theorem \ref{thm4}.

\section{Proof of theorem \ref{thm1}}\label{sec2}
If $G$ is a connected graph with order $n$, by Perron-Frobenius theorem, then there exists a positive eigenvector $X=(x_1,...,x_n)^T$ corresponding to $\rho(G)$.
Now let $G$ be an extremal graph to $\textit{spex}_\mathcal{OP}(n,F)$, and $\rho$ denote this spectral radius.
For convenience, we now normalize $X$ such that its maximum entry is 1, and then choose $u^\prime\in V(G)$ such that $x_{u^{\prime}}=\max\{x_i|i=1,2,...,n\}=1$.
For two disjoint subset $S,T\subset V(G)$, denote by $G[S,T]$ the bipartite subgraph of $G$ with vertex set $S\cup T$ that consist of all edges with one endpoint in $S$ and the other endpoint in $T.$
Set $e(S)=|E(G[S])|$ and $e(S,T)=|E\left(G[S,T]\right)|.$  Since $G$ is an outerplanar graph, we have
$e(S)\leq2|S|-3$.

Now, we are ready to give a proof of Theorem \ref{thm1}.
\begin{proof}[\textbf{proof of Theorem \ref{thm1}}]
We present the proof in a sequence of claims and start it by giving a lower bound of $\rho$.
\begin{claim}\label{claim1.1}
 $\rho \geq \sqrt {n}+ 1- \frac {n-t}{n- \sqrt {n}} > \frac{4}{5}\sqrt{n}.$
\end{claim}
\begin{proof}

Consider a wheel graph $W = K_1 \vee C_{n-1}$. Suppose that $Y = (y_1, y_2, \ldots, y_n)^T$ is the normalized Perron vector of $W$, where $y_1$ corresponds to the vertex of degree $n-1.$
By symmetry, we have $y_2 = y_3 = \cdots = y_n$. This yields the following equations:
$$
\rho(W) y_1 = (n-1) y_2, \quad \rho(W) y_2 = y_1 + 2y_2, \quad \text{and} \quad y_1^2 + (n-1)y_2^2 = 1.
$$
Solving this system of equations, we find that 
$$\rho(W) = 1 + \sqrt{n} \quad \text{and} \quad y_2^2 = \frac{1}{2(n - \sqrt{n})}.$$
Given that $\frac{n-1}{2} \leq t$, a new graph $W' \cong K_{1} \vee ((t-1)K_2 \cup (n-2t+1)K_1)$ can be obtained from $W$ by deleting $n-t$ edges in $E(C_{n-1})$.
Applying the Rayleigh Principle, we have
$$
\rho(W') \geq Y^T A(W') Y = Y^T A(W) Y - 2(n-t)y_2^2 = (1 + \sqrt{n}) - \frac{n-t}{n - \sqrt{n}}.
$$
Clearly, $W' \cong K_{1} \vee ((t-1)K_2 \cup (n-2t+1)K_1)$ is an outerplanar graph and $F$-free.
Therefore,
$$
\rho(G) \geq \rho(W') \geq \sqrt{n} + 1 - \frac{n-t}{n - \sqrt{n}} \geq \sqrt{n} + 1 - \frac{n-1}{n - \sqrt{n}} = \sqrt{n} - \frac{1}{\sqrt{n}} > \frac{4}{5}\sqrt{n},
$$
where the last inequality holds for $n > 5$, as desired.
\end{proof}

\begin{claim}\label{claim1.2}
For every $u \in V(G)$, we have $d_{G}(u)\geq nx_u-15\sqrt{n}$.
\end{claim}
\begin{proof}
Let $A := N_{G}(u)$ and $B := V(G) \setminus (A \cup \{u\})$. Then
$$
\rho^2 x_u = \sum_{y \in N(u)} \sum_{z \in N(y)} x_z \leq d_G(u) + \sum_{y \in N(u)} \sum_{z \in N_A(y)} x_z + \sum_{y \in N(u)} \sum_{z \in N_B(y)} x_z.
$$
Since $G$ is an outerplanar graph, each vertex in $A \cup B$ has at most two neighbors in $A$ (otherwise, $G$ would contain a $K_{2,3}$, which is a contradiction). Specifically,
$$
\sum_{y \in N(u)} \sum_{z \in N_A(y)} x_z \leq 2 \sum_{y \in N(u)} x_y = 2\rho x_u.
$$
Similarly, each vertex in $B$ has at most two neighbors in $A$. Therefore,
$$
\sum_{y \in N(u)} \sum_{z \in N_B(y)} x_z \leq 2 \sum_{z \in B} x_z \leq \frac{2}{\rho} \sum_{z \in B} d_G(z) \leq \frac{4e(G)}{\rho} \leq \frac{4(2n-3)}{\rho} < 10\sqrt{n},
$$
where the second-to-last inequality holds because $e(G) \leq 2n-3$ by the outerplanarity of $G$, and the last inequality holds by Claim \ref{claim1.1}.
Combining the above inequalities, we obtain
$$
\rho^2 x_u < d_u - 2\rho x_u + 10\sqrt{n}.
$$
Using Claim \ref{claim1.1} again, then
\begin{align*}
d_{u} &> \left(\sqrt{n} + 1 - \frac{n-t}{n-\sqrt{n}}\right)^2 x_u - 2\left(\sqrt{n} + 1 - \frac{n-t}{n-\sqrt{n}}\right)x_u - 10\sqrt{n} \\
      &> \left(\sqrt{n} - \frac{n-t}{n-\sqrt{n}}\right)^2 x_u - 1 - 10\sqrt{n} \\
      &> (\sqrt{n} - 2)^2 x_u - 1 - 10\sqrt{n} \quad \text{(since $\frac{n-t}{n-\sqrt{n}} < 2$ for large $n$)} \\
      &> (n - 4\sqrt{n} + 4) x_u - 1 - 10\sqrt{n} \\
      &> nx_u - 15\sqrt{n},
\end{align*}
as desired.
\end{proof}

\begin{claim}\label{claim1.3}
$x_u \leq 31/\sqrt{n}$ for any vertex $u \in V(G)\setminus{u'}$.
\end{claim}
\begin{proof}

Considering any vertex $u \in V(G) \setminus \{u'\}$, we assert that $d_G(u) < 16\sqrt{n}$.
Otherwise, there is a vertex $u \in V(G)$ with $d_G(u) \geq 16\sqrt{n}$.
Note that $x_{u'} = 1$, then $d_G(u') \geq n - 15\sqrt{n}$ by Claim \ref{claim1.2}.
Combining the above degree inequalities and the Principle of Inclusion-Exclusion, we can observe  
\begin{align*}  
|N_G(u) \cap N_G(u')| &= |N_G(u)| + |N_G(u')| - |N_G(u) \cup N_G(u')| \\  
&\geq (n - 15\sqrt{n}) + 16\sqrt{n} - n \\  
&= \sqrt{n}.  
\end{align*}  
This implies that $u$ and $u'$ share at least $\sqrt{n}$ neighbors, which yields a $K_{2,3}$ as subgraph if $n \geq 9$,
contradicting that $G$ is an outerplanar graph.
Thus, by Claim \ref{claim1.3},
\[  
16\sqrt{n} > d_u > nx_u - 15\sqrt{n},  
\]  
that is, $x_u < \frac{31}{\sqrt{n}}$.
\end{proof}

Let $A'=N(u') $ and $B'= V( G) \setminus ( N(u') \cup \{u'\} )$ for convenience.

\begin{claim}\label{claim1.4}
$\sum_{z\in B'}x_z<3488/\sqrt{n}$
\end{claim}
\begin{proof}
From the previous claim, we have $|B'| < 15\sqrt{n}$ and
\[  
\sum_{z \in B'} x_z = \frac{1}{\rho} \sum_{z \in B'} \rho x_z \leq \frac{1}{\rho} \sum_{z \in B'} \left( \frac{31}{\sqrt{n}} \right) d_z = \frac{31}{\rho \sqrt{n}} \left( e(A', B') + 2e(B') \right).  
\]  

Recall that the graph $G$ is outerplanar, then $d_{A'}(v) \leq 2$ for each vertex $v \in B'$, it implies $e(A', B') \leq 2|B'| < 30\sqrt{n}$.
Clearly, the graph $G[B']$ is also outerplanar, then $e(B') \leq 2|B'| - 3 < 30\sqrt{n}$.
Based on these bounds, the proof is completed by Claim \ref{claim1.1}.
\end{proof}

\begin{claim}\label{claim1.5}
$B'=\emptyset$.
\end{claim}
\begin{proof}
Otherwise, considering any vertex $y \in B'$ which is adjacent to at most two vertices in $A'$.
Combining Claims \ref{claim1.3} and \ref{claim1.4}, we can observe 
\begin{align*}  
\sum_{z \in N(y)} x_z &< \sum_{z \in B'} x_z + \frac{62}{\sqrt{n}}<\frac{3550}{\sqrt{n}}<1  
\end{align*}  
when $n \geq 1.27 \times 10^7$.
Let $G' = G + u' \sum_{y \in B'} y - \sum_{y \in B'} \sum_{z \in N(y)} zy$. Clearly, $G'$ is outerplanar. Next, we assert that $G'$ is $F$-free.
Suppose to the contrary that $G'$ contains a subgraph $F'$ isomorphic to $F$. Note that $G$ is $F$-free, but $G'$ contains a copy of $F$, then $V(F') \cap B'$ is not empty.
Furthermore, since $|A'| > n - 1 - 15\sqrt{n} > n - 16\sqrt{n} > |V(F)| = |V(F')|$ (as $n > [(64 + |V(F)|)^{\frac{1}{2}} + 8]^2$), it follows that
\[  
|A \backslash V(F')| = |A'| - |A' \cap V(F')| > |V(F')| - |V(F') \cap A'| \geq |V(F') \cap B'|.  
\]
Without loss of generality, let $V(F') \cap B' = \{v_1, \ldots, v_b\}$ and select $\{w_1, \ldots, w_b\} \subset A' \backslash V(F')$. Clearly, $N_{G'}(v_i) = \{u'\} \subseteq N_{G'}(w_i)$ for each $i \in \{1, \ldots, b\}$. This implies that a copy of $F$ is already present in $G$, a contradiction.
Therefore, $G'$ is $F$-free. However, by Rayleigh quotient, we obtain  
\begin{align*}  
\rho(A(G')) - \rho(A(G)) &\geq \frac{x^t(A(G') - A(G))x}{x^t x} \\  
&= \frac{2 \sum_{y \in B'} x_y}{x^t x} \left( 1 - \sum_{z \in N(y)} x_z \right) \\  
&> 0.  
\end{align*}
This contradicts the maximality of $G$. Thus, $B'= \emptyset.$
\end{proof}

Setting $u=u'$, $x_{u'}=1$ by Claim \ref{claim1.5}, the proof of Theorem \ref{thm1} (i) is well done, i.e., $G$ contains a copy of $K_{1.n-1}$.

\begin{claim}\label{claim1.6}
$G[A']$ is a union of disjoint induced paths.(In particular, we also view an isolated vertex in $G[A']$ as an induced path.)
\end{claim}
\begin{proof}
We first assert that $d_{A'}(v) \leq 2$ for arbitrary vertex $v\in V(G[A'])$. If not, then there is a $K_{2,3}$ in $G[A'\cup\{u'\}]$, a contradiction.
Next we assert that there is no cycle in $G[A']$. If not, then there is a cycle in $G[A']$. And we can contract the cycle into a triangle,
it means that there is a $K_{4}$-minor in $G$, a contradiction.
Thus, $G[A']$ is a union of disjoint induced paths.
\end{proof}
Form Claim \ref{claim1.6}, we can get Theorem \ref{thm1} (ii).
\end{proof}

\section{Proofs of Theorems \ref{thm2} and \ref{thm3}}\label{sc3}
In this section, we present the proofs of Theorems \ref{thm2} and \ref{thm3}.
The following several lemmas will be used in the sequel.
We must emphasize that some of these lemmas are adapted from the ideas of Fang, Lin, and Shi \cite{L.F. Fang-2}, and our version makes these lemmas more suitable for use in outerplanar graphs.

\begin{definition}\label{de1}
Let $s_1$ and $s_2$ be two integers with $s_1\geq s_2\geq1$, and let $H=P_{s_1}\cup P_{s_2}\cup H_0$, where $H_0$ is a disjoint union of paths. We say that $H^*$ is an $(s_1,s_2)$- transformation of $H$ if
$$H^*:=\begin{cases}P_{s_1+1}\cup P_{s_2-1}\cup H_0&\text{if }s_2\geq2,\\P_{s_1+s_2}\cup H_0&\text{if }s_2=1.\end{cases}$$
Clearly, $H^*$ is a disjoint union of paths, which implies that $K_1 \vee H^*$ is an outerplanar graph.
If $G[R]\cong H$, then we shall show that $\rho(K_1 \vee H^*)>\rho(K_1 \vee H)$ for sufficiently large $n.$
\end{definition}

\begin{lemma}\label{lm1}
Let $H$ and $H^*$ be the two graphs as shown in Definition \ref{de1}.
If $n \geq 6.5025\times2^{s_2+2}$, then $\rho(K_{1} \vee H^*) > \rho(K_{1} \vee H)$.
\end{lemma}
\begin{proof}
Note that both $K_1 \vee H$ and $K_1 \vee H^*$ are connected outerplanar graphs containing $K_{1,n-1}$ as a subgraph. Let $\rho = \rho(K_1 \vee H)$ and $\rho^* = \rho(K_1 \vee H^*)$ for convenience.
Suppose that $P' := v_1v_2 \cdots v_{s_1}$ and $P'' := w_1w_2 \cdots w_{s_2}$ are two components of $H$. If $s_2 = 1$, then $H \subset H^*$, and so $G \subset K_1 \vee H^*$.  
It follows that $\rho^* > \rho$, and the result holds.
Next, we deal with the case $s_2 = 2$. If $x_{v_1} \leq x_{w_1}$, then let $H'$ denote the graph obtained from $H$ by deleting the edge $v_1v_2$ and adding the edge $v_2w_1$.  
Clearly, $H' \cong H^*$. Moreover,  
\[  
\rho^* - \rho \geq \frac{X^T(A(K_1 \vee H^*) - A(G))X}{X^TX} \geq \frac{2}{X^TX}(x_{w_1} - x_{v_1})x_{v_2} \geq 0.  
\]  
Since $X$ is a positive eigenvector of $G$, we have $\rho x_{v_1} = 1 + x_{v_2}$.  
If $\rho^* = \rho$, then $X$ is also a positive eigenvector of $K_1 \vee H^*$, and so $\rho^* x_{v_1} = 1$, contradicting that $\rho x_{v_1} = 1 + x_{v_2}$.  
Thus, $\rho^* > \rho$, and the result holds. The case $x_{v_1} > x_{w_1}$ is similar and hence omitted here.
In what follows, we denote the vertex of degree $n-1$ by $u'$, and let $A=N(u')$ for convenience.
For $s_2\geq3$, we have the following claims.

\begin{claimm}\label{claim3.1}
$x_u\in\left[\frac{1}{\rho},\frac{1}{\rho}+\frac{2.04}{\rho^{2}}\right]$ for any vertex $u\in A$.
\end{claimm}
\begin{proof}
By Theorem \ref{thm1}, we can see $x_{u^{\prime}}=1.$ Thus
\begin{equation}\label{gongshi1}
  \rho x_{u}=x_{u^{\prime}}+\sum_{v\in N_{A}(u)}x_{v}=1+\sum_{v\in N_{A}(u)}x_{v}.
\end{equation}
Moreover, recall that $d_{A}(v)\leq2$ and $x_v \leq 1$ for any vertex $v\in A$.
Thus, $\rho x_u=\sum_{v\in N_G(u)}x_v\leq 3$,
and so $x_u\leq\frac{3}{\rho}\leq\frac{3}{\frac{4}{5}\sqrt{n}}< 0.1$ for $n > 1.41\times10^3$.
Combining this with (\ref{gongshi1}) gives   
\[  
  x_u \in \left[ \frac{1}{\rho}, \frac{1.02}{\rho} \right].  
\]  
By (\ref{gongshi1}) again, we have $\rho x_u\in\left[1,1+\frac{2.04}\rho\right]$,
which indicates that   
\[  
  x_u \in \left[ \frac{1}{\rho}, \frac{1}{\rho} + \frac{2.04}{\rho^2} \right].  
\]  
\end{proof}

Let $i$ be a positive integer, and define the boxes
\[  
A_{i} = \left[\frac{1}{\rho} - \frac{2.04 \times 2^{i}}{\rho^{2}}, \frac{1}{\rho} + \frac{2.04 \times 2^{i}}{\rho^{2}}\right]  
\quad \text{and} \quad  
B_{i} = \left[-\frac{2.02 \times 2^{i}}{\rho^{2}}, \frac{2.02 \times 2^{i}}{\rho^{2}}\right].  
\]  
\begin{claimm}\label{claim3.2} The following statements hold:  
 \begin{description}
   \item[(i)] $\rho^{i}(x_{v_{i+1}}-x_{v_{i}})\in A_{i}$ for $1\leq i\leq\lfloor\frac{s_{1}-1}{2}\rfloor$;
   \item[(ii)] $\rho^{i}(x_{w_{i+1}}-x_{w_{i}})\in A_{i}$ for $1\leq i\leq\lfloor\frac{s_{2}-1}{2}\rfloor$;
   \item[(iii)] $\rho^{i}(x_{v_{i}}-x_{w_{i}})\in B_{i}$ for $1\leq i\leq\lfloor\frac{s_{2}}{2}\rfloor$.
 \end{description}
\end{claimm}

\begin{proof}

(i)It suffices to demonstrate that for every $i \in \{1, \ldots, \lfloor \frac{s_1 - 1}{2} \rfloor\}$,
\[
\rho^i(x_{v_{j+1}} - x_{v_j}) \in \begin{cases}
A_i, & \text{if } j = i, \\
B_i, & \text{if } i + 1 \leq j \leq s_1 - i - 1.
\end{cases}
\]
We proceed with the proof by induction on $i$. Initially, we observe that
\begin{equation}\label{eq:gongshi2}
\rho x_{v_j} = x_{u'} + \sum_{v \in N_R(v_j)} x_v = \begin{cases}
1 + x_{v_2}, & \text{if } j = 1, \\
1 + x_{v_{j-1}} + x_{v_{j+1}}, & \text{if } 2 \leq j \leq s_1 - 1.
\end{cases}
\end{equation}
By Claim \ref{claim3.1}, we can deduce that
\[
\rho(x_{v_{j+1}} - x_{v_j}) = \begin{cases}
x_{v_1} + x_{v_3} - x_{v_2} \in A_1, & \text{if } j = 1, \\
(x_{v_j} - x_{v_{j-1}}) + (x_{v_{j+2}} - x_{v_{j+1}}) \in B_1, & \text{if } 2 \leq j \leq s_1 - 2.
\end{cases}
\]
Thus, the result holds for $i = 1$. Now, suppose $2 \leq i \leq \lfloor \frac{s_1 - 1}{2} \rfloor$, and assume the result holds for $i - 1$, i.e.,
\begin{equation}\label{eq:gongshi3}
\rho^{i-1}(x_{v_{l+1}} - x_{v_l}) \in \begin{cases}
A_{i-1}, & \text{if } l = i - 1, \\
B_{i-1}, & \text{if } i \leq l \leq s_1 - i.
\end{cases}
\end{equation}
Note that for each $j \in \{i, \ldots, s_1 - i - 1\}$, we have $\rho(x_{v_{j+1}} - x_{v_j}) = (x_{v_j} - x_{v_{j-1}}) + (x_{v_{j+2}} - x_{v_{j+1}})$. Hence,
\begin{equation}\label{eq:gongshi4}
\rho^i(x_{v_{j+1}} - x_{v_j}) = \rho^{i-1}(x_{v_j} - x_{v_{j-1}}) + \rho^{i-1}(x_{v_{j+2}} - x_{v_{j+1}}).
\end{equation}
If $j = i$, then by \eqref{eq:gongshi3}, $\rho^{i-1}(x_{v_j} - x_{v_{j-1}}) \in A_{i-1}$ and $\rho^{i-1}(x_{v_{j+2}} - x_{v_{j+1}}) \in B_{i-1}$. Thus, $\rho^i(x_{v_{j+1}} - x_{v_j}) \in A_i$ by \eqref{eq:gongshi4}, as desired. Similarly, if $i + 1 \leq j \leq s_1 - i - 1$, then $\rho^{i-1}(x_{v_j} - x_{v_{j-1}}) \in B_{i-1}$ and $\rho^{i-1}(x_{v_{j+2}} - x_{v_{j+1}}) \in B_{i-1}$. Thus, $\rho^i(x_{v_{j+1}} - x_{v_j}) \in B_i$ by \eqref{eq:gongshi4}. Hence, the result follows.

(ii) The proof is analogous to (i) and is omitted here.

(iii)It suffices to prove that $\rho^i(x_{v_j} - x_{w_j}) \in B_i$ for all $i \in \{1, \ldots, \lfloor \frac{s_2}{2} \rfloor\}$ and $j \in \{i, \ldots, s_2 - i\}$. We proceed by induction on $i$. Clearly,
\[
\rho x_{w_j} = x_{u'} + \sum_{w \in N_R(w_j)} x_w = \begin{cases}
1 + x_{w_2}, & \text{if } j = 1, \\
1 + x_{w_{j-1}} + x_{w_{j+1}}, & \text{if } 2 \leq j \leq s_2 - 1.
\end{cases}
\]
By Claim \ref{claim3.1}, we have
\[
\rho(x_{v_j} - x_{w_j}) = \begin{cases}
x_{v_2} - x_{w_2} \in B_1, & \text{if } j = 1, \\
(x_{v_{j+1}} - x_{w_{j+1}}) + (x_{v_{j-1}} - x_{w_{j-1}}) \in B_1, & \text{if } 2 \leq j \leq s_2 - 1.
\end{cases}
\]
Thus, the result holds for $i = 1$. Now, suppose $2 \leq i \leq \lfloor \frac{s_2}{2} \rfloor$, and assume the result holds for $i - 1$, i.e., $\rho^{i-1}(x_{v_l} - x_{w_l}) \in B_{i-1}$ for all $l \in \{i - 1, \ldots, s_2 - i + 1\}$. For each $j \in \{i, \ldots, s_2 - i\}$, we observe that $\rho(x_{v_j} - x_{w_j}) = (x_{v_{j-1}} - x_{w_{j-1}}) + (x_{v_{j+1}} - x_{w_{j+1}})$. Therefore,
\begin{equation}\label{eq:gongshi5}
\rho^i(x_{v_j} - x_{w_j}) = \rho^{i-1}(x_{v_{j-1}} - x_{w_{j-1}}) + \rho^{i-1}(x_{v_{j+1}} - x_{w_{j+1}}).
\end{equation}
By the induction hypothesis, $\rho^{i-1}(x_{v_{j-1}} - x_{w_{j-1}}) \in B_{i-1}$ and $\rho^{i-1}(x_{v_{j+1}} - x_{w_{j+1}}) \in B_{i-1}$. Hence, $\rho^i(x_{v_j} - x_{w_j}) \in B_i$ by \eqref{eq:gongshi5}, as desired.
\end{proof}

Since $n \geq 6.5025 \times 2^{s_2 + 2}$, it follows that $\rho \geq \frac{4}{5}\sqrt{n} > 2.04 \times 2^{\frac{s_2}{2} + 1}$. Combining this with Claim \ref{claim3.2}, we obtain  
\begin{equation}\label{eq:gongshi6}  
  x_{v_{i+1}} - x_{v_i} \geq \frac{1}{\rho^{i+1}} - \frac{2.04 \times 2^i}{\rho^{i+2}} > 0  
\end{equation}  
whenever $i \leq \min\left\{\frac{s_2}{2}, \left\lfloor\frac{s_1 - 1}{2}\right\rfloor\right\}$ and  
\begin{equation}\label{eq:gongshi7}  
  x_{v_{i+1}} - x_{w_i} = (x_{v_{i+1}} - x_{v_i}) + (x_{v_i} - x_{w_i}) \geq \left(\frac{1}{\rho^{i+1}} - \frac{2.04 \times 2^i}{\rho^{i+2}}\right) - \frac{2.04 \times 2^i}{\rho^{i+2}} > 0  
\end{equation}  
whenever $i \leq \min\left\{\left\lfloor\frac{s_2}{2}\right\rfloor, \left\lfloor\frac{s_1 - 1}{2}\right\rfloor\right\}$. Similarly, we can also deduce that  
\begin{equation}\label{eq:gongshi8}  
  x_{w_{i+1}} > x_{w_i} \text{ and } x_{w_{i+1}} > x_{v_i}  
\end{equation}  
whenever $i \leq \left\lfloor\frac{s_2 - 1}{2}\right\rfloor$.  
  
Recall that $s_2 \geq 3$. Let $t_1$ and $t_2$ be two positive integers such that $t_1 + t_2 = s_2 - 1$.
Let $H^*$ denote the graph obtained by deleting the edges $v_{t_1}v_{t_1+1}$ and $w_{t_2}w_{t_2+1}$ and adding the edges $v_{t_1}w_{t_2}$ and $v_{t_1+1}w_{t_2+1}$ in $H$. Then,  
\begin{equation}\label{eq:gongshi9}  
  \rho^* - \rho \geq \frac{X^T(A(K_1 \vee H^*) - A(K_1 \vee H))X}{X^TX} \geq \frac{2}{X^TX}(x_{v_{t_1+1}} - x_{w_{t_2}})(x_{w_{t_2+1}} - x_{v_{t_1}}).  
\end{equation}
If $s_2$ is odd, then we take $t_1 = t_2 = \frac{s_2 - 1}{2}$. Then, it follows from (\ref{eq:gongshi7}) and (\ref{eq:gongshi8}) that $x_{v_{t_1+1}} > x_{w_{t_2}}$ and $x_{w_{t_2+1}} > x_{v_{t_1}}$. Therefore, $\rho^* > \rho$ by (\ref{eq:gongshi9}), as desired.
If $s_2$ is even, then we only consider the case that $x_{w_{s_2/2}} \geq x_{v_{s_2/2}}$ since the proof for the case that $x_{w_{s_2/2}} < x_{v_{s_2/2}}$ is similar. Take $t_1 = \frac{s_2}{2}$ and $t_2 = \frac{s_2}{2} - 1$. Then, $x_{w_{t_2+1}} \geq x_{v_{t_1}}$ as $x_{w_{s_2/2}} \geq x_{v_{s_2/2}}$. If $s_1 = s_2$, then $s_1$ is even, and hence $x_{v_{s_1/2+1}} = x_{v_{s_1/2}}$ by symmetry, that is, $x_{v_{t_1+1}} = x_{v_{t_1}}$. If $s_1 \geq s_2 + 1$, then $x_{v_{t_1+1}} > x_{v_{t_1}}$ by (\ref{eq:gongshi6}). In both cases, we have $x_{v_{t_1+1}} \geq x_{v_{t_1}}$. Furthermore, from (\ref{eq:gongshi7}), we get $x_{v_{t_1}} > x_{w_{t_2}}$, and so $x_{v_{t_1+1}} > x_{w_{t_2}}$. Thus, $\rho^* \geq \rho$ by (\ref{eq:gongshi9}). If $\rho^* = \rho$, then $X$ is also an eigenvector of $\rho^*$, and so $\rho x_{v_{t_1}} = \rho^* x_{v_{t_1}} = 1 + x_{v_{t_1-1}} + x_{w_{t_2}}$. On the other hand, since $X$ is an eigenvector of $\rho$, we have $\rho x_{v_{t_1}} = 1 + x_{v_{t_1-1}} + x_{v_{t_1+1}}$.
Hence, $x_{v_{t_1+1}} = x_{w_{t_2}}$, which is a contradiction. Therefore, we conclude that $\rho^* > \rho$, and the result follows.
\end{proof}

\begin{lemma}\cite{J. Shu}.\label{lm2}
Let $G$ be a connected outerplanar graph on $n\geq3$ vertices. Then $\rho\left(G\right)\leq\frac32+\sqrt{n-\frac74}$.
\end{lemma}
\vspace*{2mm}
Now we give the proofs of Theorem \ref{thm2} and Theorem \ref{thm3}.

\begin{proof}[\textbf{Proof of Theorem \ref{thm2}}]

Let $G$ be the extremal graph with $\textit{spex}_{\mathcal{OP}}(n,B_{tl})$.
Noting that $F\in\{B_{tl}|l\geq3,t\geq1\}$ is a subgraph of $K_{1} \vee P_{n-1}$, and is not of $K_{1} \vee ((t-1)K_2\cup(n-2t+1)K_1)$,
then $G\cong K_{1} \vee G[A]$ and $G[A]$ is a disjoint union of paths by Theorem \ref{thm1}.
Suppose that $G[A]=\cup_{i=1}^qP_{n_i}$, where $q\geq2$ and $n_1\geq n_2\geq\cdots\geq n_q$.
Let $H$ be a disjoint union of $q$ paths. We use $n_i(H)$ to denote the order of the $i$-th longest path of $H$ for any $i\in\{1,...,q\}.$ 
  
If $t = 1$, then $B_{tl} \cong C_{l}$.   
When $l = 3$, it is evident that $G \cong K_{1, n-1}$ since $G$ is $C_3$-free; otherwise, adding any edge would create a $C_3$, which is a contradiction.  
When $l = 4$, let $J_n$ be the graph obtained from $K_1 \vee (n-1)K_1$ by embedding a maximum matching within its independent set. Nikiforov \cite{V. Nikiforov-2} and Zhai and Wang \cite{M.Q. Zhai-1} independently proved that $J_n$ is the extremal graph for $\textit{spex}(n, C_4)$ for both odd and even $n$. Clearly, $J_n$ is an outerplanar graph, thus $G \cong J_n$.  

It remains the case $l\geq 5$.

\begin{claimmm}\label{claim3.3}  
If $H$ is a disjoint union of paths, then $K_1 \vee H$ is $B_{1l}$-free if and only if $n_1(H) \leq l-2$.  
\end{claimmm}  
  
\begin{proof}  
Note that the length of the longest cycle in $K_1 \vee H$ is $n_1(H) + 1$, and $K_1 \vee P_{n_1(H)}$ contains a cycle $C_i$ for each $i \in \{3, \ldots, n_1(H) + 1\}$.  
Thus, $n_1(H) + 1 \leq l - 1$ if and only if $K_1 \vee H$ is $C_l$-free, i.e., $K_1 \vee H$ is $B_{1l}$-free.  
Therefore, the claim holds.  
\end{proof}  
  
Recall that $n_i$ (resp., $n_i(H)$) denotes the order of the $i$-th longest path of $G[A]$ (resp., $H$) for any $i \in \{1, \ldots, q\}$.  
By Claim \ref{claim3.3} and direct computation, we have $n_2 \leq n_1 \leq l - 2$, and so $6.5025 \times 2^l \geq 6.5025 \times 2^{n_2 + 2}$. Thus,  
\begin{equation}\label{gongshin}  
n \geq \max\{1.27 \times 10^7, [(64 + |V(F)|)^{1/2} + 8]^2, 6.5025 \times 2^{n_2+2}\}.  
\end{equation}  
  
\begin{claimmm}\label{claim3.4}  
$n_i = l - 2$ for $i \in \{1, 2, \ldots, q - 1\}$.  
\end{claimmm}  
  
\begin{proof}  
Note that $n_i \leq n_1 \leq l - 2$ for each $i \in \{1, \ldots, q - 1\}$.  
Now we assert that $n_i = l - 2$ for $i \in \{1, \ldots, q - 1\}$.  
Otherwise,
let $H'$ be an $(n_{i_0}, n_q)$-transformation of $G[A]$, where $i_0 = \min\{i | 1 \leq i \leq q - 1, n_i \leq l - 3\}$. Clearly, $n_1(H') = \max\{n_2, n_{i_0} + 1\} \leq l - 2$.  
By Claim \ref{claim3.3}, $K_1 \vee H'$ is $C_l$-free, i.e., $K_1 \vee H'$ is $B_{1l}$-free.  
However, by (\ref{gongshin}) and Lemma \ref{lm1}, we have $\rho(K_2 \vee H') > \rho$, contradicting that $G$ is extremal graph with $\textit{spex}_{\mathcal{OP}}(n, C_l)$.  
Therefore, the claim holds.  
\end{proof}  
  
Since $n_i = l - 2$ for $i \in \{1, \ldots, q - 1\}$ and $n_q \leq l - 2$, we can see that $G[A] \cong H_{\mathcal{OP}}(l - 2, l - 2)$. This completes the proof of Theorem \ref{thm2} (i).

For $t\geq2$ and $l\geq3$, we have the following claim.

\begin{claimmm}\label{claim3.5}  
If $H$ is a disjoint union of paths and $K_1 \vee H$ is $B_{tl}$-free, then $n_1 \leq (l - 1)(t - 1) + l - 2$.  
\end{claimmm}  
  
\begin{proof}  
We assert that $P_{n_1}$ contains at most $(l - 1)(t - 1) + l - 3$ edges.  
Indeed, if $P_{n_1}$ contains at least $(l - 1)(t - 1) + l - 2$ edges, then $G[A]$ must have $t$ independent $(l - 1)$-paths, say $P^{(1)}_{l-1} = u_1u_2\cdots u_{l - 1}, \ldots, P^{(t)}_{l-1} = u_{(t - 1)(l - 1) + 1}\cdots u_{t(l - 1)}$.  
Then $u, u_1, \ldots, u_{t(l - 1)}$ generate $t$ edge-disjoint $C_l$ which only intersect at a common vertex $u$, so $G$ contains a copy of $B_{tl}$, a contradiction.  
\end{proof}  
  
With Claim \ref{claim3.5}, we are ready to complete the proof of Theorem \ref{thm2} (ii).  
We first assert that if $P_{n_1} = P_{(t - 1)(l - 1) + l - 2}$, then $n_2 \leq l - 2$. Suppose to the contrary that $n_2 \geq l - 1$, then $G[A]$ has $t$ independent $(l - 1)$-paths.  
Thus $G[A \cup \{u\}]$ would generate $t$ edge-disjoint $C_l$ which only intersect at the common vertex $u$, so $G$ contains a copy of $B_{tl}$, a contradiction.
Let $H^* = H_{\mathcal{OP}}(tl - t - 1, l - 2)$. Clearly, $K_1 \vee H^* = K_1 \vee H_{\mathcal{OP}}((t - 1)(l - 1) + l - 2, l - 2)$.  
If $G[A] \cong H^*$, then $G \cong K_1 \vee H_{\mathcal{OP}}(tl - t - 1, l - 2)$, and we are done.  
If $G[A] \ncong H^*$, by Theorem \ref{thm1} and Claim \ref{claim3.5}, then $G[A]$ is a disjoint union of some paths and $n_1 \leq (l - 1)(t - 1) + l - 2$,  
and so we can obtain $H^*$ by applying a series of $(s_1, s_2)$-transformations to $G[A]$, where $s_2 \leq s_1 \leq (t - 1)(l - 1) + l - 2$.  
As $n \geq 6.5025 \times 2^{(t - 1)(l - 1) + l} \geq 6.5025 \times 2^{s_2 + 2}$, by Lemma \ref{lm2}, we conclude that $\rho < \rho(K_1 \vee H_{\mathcal{OP}}((t - 1)(l - 1) + l - 2, l - 2)) = \rho(K_1 \vee H^*)$, which is impossible by the maximality of $\rho = \rho(G)$.  
This completes the proof.
\end{proof}

\begin{proof}[\textbf{Proof of Theorem \ref{thm3}}]

Let $G$ be the extremal graph with $\textit{spex}_{\mathcal{OP}}(n,(t+1)K_{2})$.
If $t=1$, then $G\cong K_{1,n-1}\cong \textit{SPEX}_{\mathcal{OP}}(n,2K_{2})$.
It remains that $t\geq 2$, we consider the following cases.
\begin{case}{$G$ is a connected graph.}\label{case3.1}

Noting that $F\in\{(t+1)K_{2}|k\geq2\}$ is a subgraph of $K_{1} \vee P_{n-1}$, and is not of $K_{1} \vee ((t-1)K_2\cup(n-2t+1)K_1)$,
then $G\cong K_{1} \vee G[A]$ where $G[A]$ is a disjoint union of paths by Theorem \ref{thm1}.
We assert that $G[A]$ contains at most $2t - 2$ edges. Otherwise, $G[A]$ contains at least $2t - 1$ edges. Then $G[A]$ must have $t$ independent edges, say $u_1u_2, \ldots, u_{2t-1}u_{2t}$.
Take $\bar{u} \in A \setminus \{u_1, \ldots, u_{2t}\}$.
Then $u, \bar{u}, u_1, \ldots, u_{2t}$ would generate $t + 1$ independent edges, and so $G$ contains a copy of $(t + 1)K_2$, a contradiction.  
Let $H^* = H_{\mathcal{OP}}(2t - 1, 1)$. Clearly, $K_1 \vee H^* = K_1 \vee (P_{2t-1} \cup (n - 2t)K_1)$. If $G[A] \cong H^*$, then $G \cong K_1 \vee H^*$, and we are done.
If $G[A] \ncong H^*$, by Theorem \ref{thm1},
then $G[A]$ is a disjoint union of some paths and contains at most $2t - 2$ edges.
Thus, we can obtain $H^*$ by applying a series of $(s_1, s_2)$-transformations to $G[A]$, where $s_2 \leq s_1 \leq 2t - 1$.
As $n \geq 6.5025 \times 2^{2t + 1} \geq 6.5025 \times 2^{s_2 + 2}$, by Lemma \ref{lm1}, we conclude that $\rho < \rho(K_1 \vee H_{\mathcal{OP}}(2t - 1, 1)) = \rho(K_1 \vee H^*)$,
which is impossible by the maximality of $\rho = \rho(G)$.

\end{case}
\begin{case}{$G$ is a disconnected graph.}\label{case3.2}

Let $n \geq N + \frac{3}{2} + 3\sqrt{N - \frac{7}{4}}$, where $N = \max\{1.27 \times 10^{7}, 6.5025 \times 2^{s_2+2}\}$.
Suppose that $G$ is disconnected with components $G_1, G_2, \ldots, G_l$ ($l \geq 2$), and without loss of generality, assume $\rho(G_1) = \rho(G) = \rho$.
For each $i \in \{1, \ldots, l\}$, let $V(G_i) = \{u_{i,1}, u_{i,2}, \ldots, u_{i,n_i}\}$ and $t_i$ be the matching number of $G_i$.
Clearly, $\sum_{i=1}^l t_i \leq t$ and $\sum_{i=1}^l n_i = n$.
If $t_1 \leq t-1$, then constructing $G'$ by removing all edges in $G_2, \ldots, G_l$ and connecting each vertex in these components to $u_{1,1}$ in $G_1$ results in a connected $(t+1)K_2$-free outerplanar graph.
Since $G_1$ is a proper subgraph of $G'$, $\rho(G') > \rho(G_1) = \rho(G)$, contradicting the maximality of $\rho(G)$. Thus, $t_1 = t$ and $G = G_1 \cup (n-n_1)K_1$.
Furthermore, $G_1$ must be extremal in SPEX$_\mathcal{OP}(n_1, (t+1)K_2)$, otherwise $G$ would not be extremal in $SPEX_\mathcal{OP}(n, (t+1)K_2)$.
If $n_1 \geq N$, then $G_1 \cong K_1 \vee (P_{2t-1} \cup (n_1-2t)K_1)$ by Case \ref{case3.1}.
Therefore, $G$ is a proper subgraph of $K_1 \vee (P_{2t-1} \cup (n-2t)K_1)$, again contradicting the maximality of $\rho(G)$.
Next, we consider $n_1 < N$. For $t \geq 2$, by Lemma \ref{lm2}, we obtain
$$\rho=\rho(G_1)\leq \frac{3}{2}+\sqrt{n_1-\frac{7}{4}} < \frac{3}{2}+\sqrt{N-\frac{7}{4}} \leq \sqrt{n-1}=\rho(K_{1, n-1})\leq\rho(M_t)$$
as $n\geq N+\frac{3}{2}+3\sqrt{N-\frac{7}{4}}$ and $K_{1,n-1}$ is a subgraph of $M_t.$ This contradicts the maximality of $\rho(G)$.
For $t=1$, $G$ must be a proper subgraph of $K_{1,n-1} = M_t$, leading to a similar contradiction.
Therefore, G must be connected. By Case \ref{case3.1}, we conclude that $G \cong M_t$, which completes the proof.
\end{case}
\end{proof}

\section{Proof of theorem \ref{thm4}}\label{sc4}
In this section, we determine the SPEX$_{\mathcal{P}}(n,B_{tl})$ for $t\geq3$ and $l\geq3$ by proofing Theorem \ref{thm4}.
Prior to this, we first list some useful lemmas from \cite{X.L. Wang}.

\begin{lemma}\cite{X.L. Wang}.\label{lm3}  
Let $F$ be a planar graph not contained in $K_{2,n-2}$ where $n \geq \max\{2.67 \times 9^{17}, \frac{10}{9}|V(F)|\}$.  
Suppose that $G$ is a connected extremal graph in $\mathrm{SPEX}_{\mathcal{P}}(n, F)$ and
$X$
is the positive eigenvector of $\rho := \rho(G)$ with $\max_{v \in V(G)} x_{v} = 1$.   
Then the following two statements hold.  
 \begin{description}
   \item[(i)] There exist two vertices $u',u''\in V( G)$ such that $R:=N_G(u')\cap N_G( u'')=V( G)\setminus\{u', u''\}$ and $x_{u'}= x_{u''}= 1.$ 
In particular, $G$ contains a copy of $K_{2, n-2}.$
   \item[(ii)] The subgraph $G[R]$ of $G$ induced by $R$ is a disjoint union of some paths and cycles.
Moreover, if $G[R]$ contains a cycle then it is exactly a cycle, i.e., $G[R] \cong C_{n-2}$, and
if $u'u''\in E(G)$ then $G[R]$ is a disjoint union of some paths.
 \end{description}
 \end{lemma}
 
\begin{lemma}\cite{X.L. Wang}.\label{lm4}
Suppose further that G contains $K_{2, n- 2}$ as a subgraph. Let $u_{1}, u_{2}$ be the two vertices of $G$ that have degree $n-2$ in $K _{2, n- 2}$.
For any vertex $u\in V( G) \setminus \{ u_1, u_2\} $, we have 
$$\frac 2\rho \leq x_u\leq \frac 2\rho + \frac {4. 496}{\rho ^2}.$$
\end{lemma}
 
 \begin{definition}\cite{X.L. Wang}.\label{de2}
Let $s_1$ and $s_2$ be two integers with $s_1\geq s_2\geq1$, and let $H=P_{s_1}\cup P_{s_2}\cup H_0$, where $H_0$ is a disjoint union of paths. We say that $H^*$ is an $(s_1,s_2)$-transformation of $H$ if
$$H^*:=\begin{cases}P_{s_1+1}\cup P_{s_2-1}\cup H_0&\text{if }s_2\geq2,\\P_{s_1+s_2}\cup H_0&\text{if }s_2=1.\end{cases}$$
Clearly, $H^*$ is a disjoint union of paths, which implies that $K_2 \vee H^*$ is planar.
If $G[R]\cong H$, then we shall show that $\rho(K_2 \vee H^*)>\rho(K_2 \vee H)$ for sufficiently large $n.$
\end{definition}

\begin{lemma}\cite{X.L. Wang}.\label{lm5}
Let $H$ and $H^*$ be the two graphs as shown in Definition 2.
When $n\geq\max\{2.67\times9^{17},10.2\times2^{(t-1)(l-1)+l-3}+2\}$, we have $\rho(K_{2} \vee H^*) > \rho(K_{2} \vee H)$.
\end{lemma}

\begin{proof}[\textbf{Proof of Theorem \ref{thm4}}]
With Lemmas \ref{lm3} and \ref{lm4}, we have the following claim.
\begin{claimmmm}\label{claim4.1}  
The vertices $u'$ and $u''$ are adjacent. Furthermore, the subgraph $G[R]$ of $G$ induced by $R$ is a disjoint union of some paths.  
\end{claimmmm}  
  
\begin{proof}  
Suppose to the contrary that $u'u'' \notin E(G)$. Then we assert that $G[A]$ contains at most $t - 1$ independent $(l - 1)$-paths.
If not, $G$ would contain $t$ edge-disjoint $C_l$ that intersect in one common vertex $u'$, which is impossible as $G$ is $B_{tl}$-free. We denote the edges of these independent $(l - 1)$-paths by $E(IP)$.
Let $G' = G + u'u'' - \sum_{u_iu_j \in E(IP)} u_iu_j$. Clearly, $G'$ is planar and $B_{tl}$-free. By Lemma \ref{lm4}, we have  
\[  
\begin{aligned}  
\rho(G') - \rho &\geq \frac{2}{X^TX} \left( x_{u'}x_{u''} - \sum_{u_iu_j \in E(IP)} x_{u_i}x_{u_j} \right) \\  
&\geq \frac{2}{X^TX} \left( 1 - (t - 1)(l - 2) \cdot \left( \frac{2}{\rho} + \frac{4.496}{\rho^2} \right)^2 \right) \\  
&\geq \frac{2}{X^TX} \left( 1 - (t - 1)(l - 2) \cdot \left( \frac{2}{\sqrt{2n - 4}} + \frac{4.496}{2n - 4} \right)^2 \right) \\  
&> 0,  
\end{aligned}  
\]  
as $\rho \geq \sqrt{2n - 4}$ and $n \geq \max\{2.67 \times 9^{17}, 10.2 \times 2^{(t - 1)(l - 1) + l - 3} + 2\}$.
Hence, $\rho(G') > \rho$, contrary to the maximality of $\rho = \rho(G)$. Therefore, $u'u'' \in E(G)$.
From Lemma \ref{lm3}, we establish that the subgraph $G[R]$ of $G$ induced by $R$ is a disjoint union of some paths.  
\end{proof}

From the Claim \ref{claim4.1}, we can assume that $G[R]=\cup_{i=1}^qP_{n_i}$, where $q\geq2$ and $n_1\geq n_2\geq\cdots\geq n_q.$
Let $H$ be a union of disjoint paths.
We use $n_i(H)$ to denote the order of the $i$-th longest path of $H$ for any $i\in\{1,...,q\}.$ 

 \begin{claimmmm}\label{claim4.2}
If $H$ is a union of disjoint paths and $K_2 \vee H$ is $B_{tl}$-free,
then $n_1+n_2\leq (t-1)(l-1)+l-3$.
\end{claimmmm}

\begin{proof}  
Suppose to the contrary that $n_1 + n_2 \geq (t - 1)(l - 1) + l - 2$.  
Let $a$ and $b$ be the maximum number of independent $(l - 1)$-paths in $P_{n_1}$ and $P_{n_2}$, respectively, i.e.,  
\[  
a = \left\lfloor \frac{n_1}{l - 1} \right\rfloor \quad \text{and} \quad b = \left\lfloor \frac{n_2}{l - 1} \right\rfloor.  
\]  
This implies that $n_1 - a(l - 1) \leq l - 2$ and $n_2 - b(l - 1) \leq l - 2$.
Next, we assert that $a + b = t - 1$. Clearly, if $a + b > t - 1$, then $G$ would contain a $B_{tl}$, which is a contradiction.  
If $a + b < t - 1$, recall that $n_1 + n_2 \geq (t - 1)(l - 1) + l - 2$, then 
\[  
n_1 + n_2 - a(l-1) - b(l-1) \geq 2l - 3.  
\]  
On the other hand, since $n_1 - a(l-1) \leq l - 2$ and $n_2 - b(l-1) \leq l - 2$, we have  
\[  
n_1 + n_2 - a(l-1) - b(l-1) \leq 2l - 4 < 2l - 3,  
\]
which is a contradiction. Thus, $a + b = t - 1$. Therefore, $G[P_{n_1} \cup P_{n_2} \cup \{u', u''\}]$ would generate a $B_{tl}$ due to $n_1 + n_2 \geq (t - 1)(l - 1) + l - 2$, a contradiction.   
\end{proof}

With Claims \ref{claim4.1} and \ref{claim4.2}, we are ready to complete the proof of Theorem \ref{thm4}.
Let $H^* = H_{\mathcal{P}}(tl-t-l, l-2)$. Clearly, $K_2 + H^* = K_2 \vee H_{\mathcal{P}}((t-2)(l-1)+l-2, l-2)$, which means that $n_1 = (t-2)(l-1)+l-2$ and $n_2 = l-2$.  
If $G[R] \cong H^*$, then $G \cong K_2 \vee H_{\mathcal{P}}(tl-t-l, l-2)$, and we are done.
If $G[R] \ncong H^*$, by Lemma \ref{lm3} and Claim \ref{claim4.2}, then $G[R]$ is a disjoint union of some paths and $n_1 + n_2 \leq (l-1)(t-1) + l-3$.
Now, we shall consider the quantity $r$ which represents the number of independent $(l-1)$-paths in $P_{n_1}$. 
Firstly, we can get $r \leq t-2$ for $l = 3,4$ and $r \leq t-1$ for $l \geq 5$. Otherwise, there would be $B_{tl}$ in $G$, a contradiction.
Next, we will follow the following cases.
  
\begin{casee}\label{case4.1}  
If $r \leq t-2$, then we can obtain $H^*$ by applying a series of $(s_1, s_2)$-transformations to $G[R]$, where $s_2 \leq s_1 \leq (t-1)(l-1) + l-3$.  
As $n \geq 10.2 \times 2^{(t-1)(l-1)+l-3} + 2 \geq 10.2 \times 2^{s_2} + 2$, we conclude that $\rho < \rho(K_2 \vee H^*) = \rho(K_2 \vee H_{\mathcal{P}}((t-2)(l-1)+l-2, l-2))$ by Lemma \ref{lm5},  
which is impossible by the maximality of $\rho = \rho(G)$.
Up to this case, we can conclude that $G\cong K_2 \vee H_{\mathcal{P}}((t-2)(l-1)+l-2, l-2))$ for $l=3,4$.
\end{casee}  
  
\begin{casee}  
If $r = t-1$, then we have the following claims.
\begin{claimmmm}\label{claim4.3}  
Let $H$ be an union of $q$ vertex–disjoint paths and $\bar{n}_1=n_1-(t-1)(l-1)$.
Then $K_2 \vee H$ does not contain $B_{tl}$ if and only if $\bar{n}_1+n_2 \leq l-3$ or $n_2+n_3 \leq l-3$
\end{claimmmm}

\begin{proof}
We denote that $P_{n_1} = v_1v_2 \cdots v_{n_1}$, $P' = v_1v_2 \cdots v_{(t-1)(l-1)}$ and $P_{\bar{n}_1} = v_{(t-1)(l-1)+1} \cdots v_{n_1}$.
Observing that the longest cycle which belongs to $K_2 \vee (H-P')$ and intersects with $B_{(t-1)l}$ only at ${u'}$ is of order $\bar{n}_1 + n_2+2$ or $n_2+n_3+2$.
Furthermore, $K_2 \vee (P_{\bar{n}_1} \cup P_{n_2})$ or $K_2 \vee (P_{n_2} \cup P_{n_3})$ contains a cycle $C_i$ for each $i \in \{3, …,\max\{\bar{n}_1+ n_2 + 2, n_2+n_3 +2\}\}.$
Therefore, $\bar{n}_1+n_2 \leq l-3$ or $n_2+n_3 \leq l-3$ if and only if $K_2 \vee H$ is $B_{tl}$‐free.
Hence, the claim holds.
\end{proof}

\begin{claimmmm}\label{claim4.4}  
$n_i=n_2$ for $i \in \{3,\ldots,q-1\}$.
\end{claimmmm}

\begin{proof}
Suppose to the contrary, then there exists an $i_0 = \min\{i|3 \leq i \leq q-1\}$.
If $i_0 > 3$, then let $H'$ be an $(n_{i_0},n_q)$-transformation. Thus, $n_1(H')=n_1$(it means $\bar{n}_1(H') = \bar{n}_1$), $n_2(H')=\max\{n_2,n_{i_0+1}\}$ and $n_3(H')=n_3$. Therefore, $\bar{n}_1(H')+n_2(H') \leq l-3$ or $n_2(H')+n_3(H') \leq l-3$. By Claim \ref{claim4.3}, $K_2 \vee H'$ is $B_{tl}$-free. However, by Lemma \ref{lm5}, we have $\rho(K_2 \vee H') > \rho$, a contradiction. If $i_0 =3$, we first suppose that there exists a $j$ such that $n_j < n_3 \leq n_2-1$, where $4 \leq j \leq q-1$. Let $H''$ be the $(n_j,n_t)$-transformation of $G[R]$, where $4 \leq j \leq q-1.$ Then $\bar{n}_{1}(H'') = \bar{n}_{1}$, $n_{2}(H'') = n_{2}$ and $n_{3}(H'')=\max\{n_{3},n_{j}+1\}=n_{3}.$ Using the methods similar to those for $i_0>3$, we can also get a contradiction. Then we suppose that $n_j=n_3$ for any $j\in\{4,\cdots,q-1\}$. If $\bar{n}_1+n_2 \leq l-3$, let $H^{(3)}$ be the $(n_3,n_q)$-transformation of $G[R]$, then $n_{1}(H^{(3)})=n_{1},n_{2}(H^{(3)})=\max\{n_{2},n_{3}+1\}=n_{2}.$ Thus we have $\bar{n}_{1}(H^{(3)})+n_{2}(H^{(3)}) \leq l-3$. If $n_2+n_3 \leq l-3$, let $H^{(4)}$be the $(n_1,n_2)$-transformation of $G[R]$, then $n_1(H^{(4)})=n_1+1$, $n_{2}(H^{(4)})=\max\{n_{2}-1,n_{3}\}=n_{2}-1$ and $n_{3}(H^{(4)})=n_{3}.$ Thus we get $n_2(H^{(4)})+n_3(H^{(4)})=n_2+n_3-1 \leq l-4<l-3.$ It follows from Claim \ref{claim4.3} that neither $K_2+H^{(3)}$ nor $K_2+H^{(4)}$ contains $B_{tl}.$ By Lemma \ref{lm5},we get $\rho(K_{2} \vee H^{(3)})>\rho(G)$ and $\rho(K_{2} \vee H^{(3)})>\rho(G)$, contraction.
\end{proof}

Now we assert that $G[R] \cong H_{\mathcal{P}}((t-1)(l-1)+l-3-x, x)$ for a fixed integer $x \in [1, \lfloor\frac{l-3}{2}\rfloor]$, and then $G \cong K_2 \vee H_{\mathcal{P}}((t-1)(l-1)+l-3-x, x)$.  
Otherwise, we can obtain $H_{\mathcal{P}}((t-1)(l-1)+l-3-x, x)$ by applying a series of $(s_1, s_2)$-transformations to $G[R]$, where $s_2 \leq s_1 \leq (t-1)(l-1) + l-4$.  
As $n \geq 10.2 \times 2^{(t-1)(l-1)+l-3} + 2 \geq 10.2 \times 2^{s_2} + 2$, we have that $\rho < \rho(K_2 \vee H_{\mathcal{P}}((t-1)(l-1)+l-3-x, x)$ by Lemma \ref{lm5}, which is impossible by the maximality of $\rho = \rho(G)$.

Next, we assert that $\rho(K_2 \vee H_{\mathcal{P}}((t-1)(l-1)+l-3-x, x) < \rho(K_2 \vee H^*) = \rho(K_2 \vee H_{\mathcal{P}}((t-2)(l-1)+l-2, l-2))$.  
We use $P^{(i)}$ and $n_i(H)$ to denote the $i$th longest path of $H_{\mathcal{P}}((t-1)(l-1)+l-3-x, x)$ and the order of the $i$-th longest path of $H_{\mathcal{P}}((t-1)(l-1)+l-3-x, x)$ for any $i \in \{1, \ldots, q\}$, respectively. Assuming that $P^{(1)}=v_1v_2 \cdots v_{(t-1)(l-1)+l-3-x}$.
Since $t \geq 3$, $l \geq 5$ and $n_1 = (t-1)(l-1)+l-3-x$, there exists an edge $w'w''=v_{(t-2)(l-1)+l-2}v_{(t-1)(l-1)} \in E(P^{(1)})$. 
Note that $n-2 = \sum_{i=1}^{q} n_i \leq n_1 + (q-1)n_2$, it implies that $q-2 \geq n_3 + 1$ by $n \geq 10.2\times2^{(t-1)(l-1)+l-4}+2 > (t-1)(l-1)+2(l-3)+{\lfloor\frac{l-3}{2}\rfloor}^2+ 2$.  
Then $P^{(2)}, P^{(3)}, \ldots, P^{(n_3+1)}$ are paths of order $n_2$.  
And assume that $P^{(q)} = w_1w_2 \cdots w_{n_3}$.
Let $G'$ be obtained from $G$ by (i) deleting $w'w''$ and joining $w''$ to an endpoint of $P^{(n_3+2)}$ and (ii) deleting all edges of $P^{(q)}$ and joining $w_i$ to an endpoint of $P^{(i+1)}$ for each $i \in \{1, \ldots, n_3\}$, 
it means that $G'$ is obtained from $G$ by deleting $n_3$ edges and adding $n_3 + 1$ edges.  
By Lemma \ref{lm4},  
\[  
\frac{4}{\rho^2} \leq x_{u_i}x_{u_j} \leq \frac{4}{\rho^2} + \frac{24}{\rho^3} + \frac{36}{\rho^4} < \frac{4}{\rho^2} + \frac{25}{\rho^3}  
\]  
for any vertices $u_i, u_j \in R$. Therefore,  
\[  
\rho(G') - \rho \geq \frac{X^T(A(G') - A(G))X}{X^TX} > \frac{2}{X^TX} \left( \frac{4(n_3 + 1)}{\rho^2} - \frac{4n_3}{\rho^2} - \frac{25n_3}{\rho^3} \right) > 0,  
\]  
where $n_3 \leq \lfloor\frac{l-3}{2}\rfloor \leq \frac{4}{25} \sqrt{2n-4} \leq \frac{4}{25} \rho$ as $n \geq \frac{625}{32} \lfloor\frac{l-3}{2}\rfloor^2 + 2$.
However, $P_{n_1}$ in $G'$ has only $(t-2)$ independent $(l-1)$-paths, which means that $r \leq t-2$. Then, by Case \ref{case4.1},
we have $\rho(K_2 \vee H_{\mathcal{P}}((t-1)(l-1)+l-3-x, x) = \rho < \rho(G') \leq \rho(K_2 \vee H_{\mathcal{P}}((t-2)(l-1)+l-2, l-2))$,  
which is impossible by the maximality of $\rho = \rho(G)$, leading to a contradiction.  
\end{casee}
\end{proof}

\section{Acknowledgement}
We would like to show our great gratitude to anonymous referees for their valuable suggestions which greatly improved the quality of this paper.


\begin{thebibliography}{99}
\bibitem{B. Bollobás} B. Bollobás, V. Nikiforov, Cliques and the spectral radius, \emph{J. Comb. Theory, Ser. B.} \textbf{97} (2007) 859-865.

\vspace{-2mm}
\bibitem{B. N. Boots} B. N. Boots and G. F. Royle, A conjecture on the maximum value of the principal eigenvalue of a planar graph, \emph{Geogr. Anal.} \textbf{23} (1991) 276-282.

\vspace{-2mm}
\bibitem{R.A. Brualdi} R.A. Brualdi, E.S. Solheid, On the spectral radius of complementary acyclic matrices of zeros and ones, \emph{SIAM J. Algebraic Discrete Method.} \textbf{7(2)} (1986) 265-272.

\vspace{-2mm}
\bibitem{D. Cao} D. Cao and A. Vince, The spectral radius of a planar graph, \emph{Linear Algebra Appl.} \textbf{187} (1993), 251-257.

\vspace{-2mm}
\bibitem{M.Z. Chen-1} M.Z. Chen, A.M. Liu, X.-D. Zhang, On the spectral radius of graphs without a star forest, \emph{Discrete Math.} \textbf{344 (4)} (2021) 112269.

\vspace{-2mm}
\bibitem{M.Z. Chen-2} M.Z. Chen, A.M. Liu, X.-D. Zhang, Spectral extremal results with forbidding linear forests, \emph{Graphs Comb.} \textbf{35} (2019) 335-351.

\vspace{-2mm}
\bibitem{M.Z. Chen-3} M.Z. Chen, X.-D. Zhang, Some new results and problems in spectral extremal graph theory (in Chinese), \emph{J. Anhui Univ. Nat. Sci.} \textbf{42} (2018) 12-25.

\vspace{-2mm}
\bibitem{S. Cioab˘a-1} S. Cioabă, D.N. Desai, M. Tait, The spectral even cycle problem, \emph{Comb. Theory} \textbf{1} (2024) Paper No. 10, 17.

\vspace{-2mm}
\bibitem{S. Cioab˘a-2} S. Cioabă, D.N. Desai, M. Tait, The spectral radius of graphs with no odd wheels, \emph{European J. Combin.} \textbf{99} (2022), Paper No. 103420, 19 pp.

\vspace{-2mm}
\bibitem{D. Cvetković} D. Cvetković and P. Rowlinson, The largest eigenvalue of a graph: A survey, \emph{Linear Multilinear Algebra} \textbf{28} (1990), no. 1-2, 3-33.

\vspace{-2mm}
\bibitem{L.F. Fang-1} L.F. Fang, Lin H, Shu J, Zhang Z. Spectral extremal results on trees, \emph{Electron. J. Combin.} \textbf{31(2)} (2024), Paper No. 2.34, 17.

\vspace{-2mm}
\bibitem{L.F. Fang-2} L.F. Fang, H.Q. Lin and Y.T Shi, Extremal spectral results of planar graphs without vertex‐disjoint cycles, \emph{J. Graph Theory} \textbf{106.3} (2024): 496-524.

\vspace{-2mm}
\bibitem{L.H. Feng} L.H. Feng, G.H. Yu, X.-D. Zhang, Spectral radius of graphs with given matching number, \emph{Linear Algebra Appl.} \textbf{422} (2007) 133-138.

\vspace{-2mm}
\bibitem{Y.T. Li} Y.T. Li, W.J. Liu, L.H. Feng, A survey on spectral conditions for some extremal graph problems, \emph{Adv. Math.} \textbf{51 (2)} (2022) 193-258.
 
\vspace{-2mm}
\bibitem{H.Q. Lin} H.Q. Lin and B. Ning, A complete solution to the Cvetković‐Rowlinson conjecture, \emph{J. Graph Theory.} \textbf{97} (2021) 441-450.

\vspace{-2mm}
\bibitem{V. Nikiforov-1} V. Nikiforov, A spectral condition for odd cycles in graphs, \emph{Linear Algebra Appl.} \textbf{428 (7)} (2008), 1492-1498.

\vspace{-2mm}
\bibitem{V. Nikiforov-2} V. Nikiforov, Bounds on graph eigenvalues II, \emph{Linear Algebra Appl.} \textbf{427 (2-3)} (2007), 183-189.

\vspace{-2mm}
\bibitem{V. Nikiforov-3} V. Nikiforov, Some new results in extremal graph theory, Surveys in Combinatorics, \emph{London Math. Soc. Lec. Note Ser.} \textbf{392}(2011)141-181.

\vspace{-2mm}
\bibitem{V. Nikiforov-4} V. Nikiforov, The spectral radius of graphs without paths and cycles of specified length, \emph{Linear Algebra Appl.} \textbf{432 (9)} (2010) 2243-2256.

\vspace{-2mm}
\bibitem{A. J. Schwenk} A. J. Schwenk and R. J. Wilson, \emph{Chapter 11: On the eigenvalues of a graph}, Selected topics in graph theory (L. W. Beineke, and R. J. Wilson, eds.), Academic Press, London, 1978, pp. 307-336.

\vspace{-2mm}
\bibitem{J. Shu} J. Shu and Y. Hong, Upper bounds for the spectral radii of outerplanar graphs and Halin graphs, \emph{Ann. Math. Ser. A} \textbf{21} (2000), no. 6, 677-682 (Chinese).

\vspace{-2mm}
\bibitem{M. Tait} M. Tait and J. Tobin, Three conjectures in extremal spectral graph theory, \emph{J. Combin. Theory, Ser. B} \textbf{126} (2017), 137-161.

\vspace{-2mm}
\bibitem{X.L. Wang} X.L. Wang, X.Y. Huang and H.Q. Lin, On the spectral extremal problem of planar graphs, arxiv:2402.16419, 2024.

\vspace{-2mm}
\bibitem{H. Wilf} H. Wilf, Spectral bounds for the clique and independence numbers of graphs,\emph{ J. Comb. Theory, Ser. B} \textbf{40} (1986) 113-117.

\vspace{-2mm}
\bibitem{M.Q. Zhai-1} M.Q. Zhai, B. Wang, Proof of a conjecture on the spectral radius of C4-free graphs, \emph{Linear Algebra Appl.} \textbf{437 (7)} (2012), 1641-1647.

\vspace{-2mm}
\bibitem{M.Q. Zhai-2} M.Q. Zhai, H.Q. Lin, Spectral extrema of graphs: forbidden hexagon, \emph{Discrete Math.} \textbf{343 (10)} (2020), Paper No. 112028, 6 pp.

\vspace{-2mm}
\bibitem{M.Q. Zhai-3} M.Q. Zhai, and M.H. Liu, Extremal problems on planar graphs without k edge-disjoint cycles, \emph{Adv. in Appl. Math.} \textbf{157} (2024): 102701.

\vspace{-2mm}
\bibitem{H.R. Zhang} H.R. Zhang and W.H. Wang, Extremal spectral results of planar graphs without  $ C_ {l, l} $ or $\mathrm {Theta} $ graph, arxiv:2403.10163, 2024.

\vspace{-2mm}
\bibitem{Y.H. Zhao} Y.H. Zhao, X.Y. Huang, H.Q. Lin, The maximum spectral radius of wheel-free graphs, \emph{Discrete Math.} \textbf{344 (5)} (2021), Paper No. 112341, 13 pp.

\end{thebibliography}
\end{document}